\documentclass[11pt]{article}

\usepackage[a4paper,margin=1in]{geometry}
\usepackage{amsmath,amssymb,amsthm}
\usepackage{mathtools}
\usepackage{microtype}
\usepackage[hidelinks]{hyperref}
\usepackage{xcolor}
\usepackage{dsfont}
\usepackage{bbm}
\usepackage{mathrsfs}
\usepackage{amsthm}

\newtheorem{theorem}{Theorem}[section]
\newtheorem{lemma}[theorem]{Lemma}
\newtheorem{corollary}[theorem]{Corollary}
\newtheorem*{conjecture}{Conjecture}

\theoremstyle{definition}
\newtheorem{remark}[theorem]{Remark}
\newtheorem{definition}[theorem]{Definition}
\newtheorem{example}[theorem]{Example}

\newtheorem{proposition}[theorem]{Proposition}

\newcommand{\EE}{\mathbb E}
\newcommand{\N}{\mathbb N}
\newcommand{\dto}{\xRightarrow{d}}
\newcommand{\R}{\mathbb R}

\newcommand{\PP}{\mathbb P}
\newcommand{\pto}{\xrightarrow{\PP}}
\newcommand{\II}{\mathcal I}
\newcommand{\ISSV}{\mathrm{ISSV}}

\newcommand{\ind}{\mathbbm{1}}
\newcommand{\Law}{\mathcal L}

\title{Subsequential Version of Breiman's Conjecture}
\author{Volodymyr PAVLENKOV\textsuperscript{1}}
\date{}

\begin{document}
\maketitle

\begin{flushright}
\itshape
In loving memory of my mother,\\
Tetyana M. Pavlenkova (1955--2026).
\end{flushright}
\vspace{0.5cm}

\let\thefootnote\relax\footnotetext{\textsuperscript{1}Department of Mathematics, University of York,
Heslington, York, YO10 5DD, England, {\texttt{volodymr.pavlenkov@york.ac.uk}}.}

\begin{abstract}
We study randomly weighted sums
\[
T_n=\frac{\sum_{i=1}^n X_iY_i}{\sum_{i=1}^n Y_i},
\]
where $(X_i)$ and $(Y_i)$ are independent sequences
of i.i.d.\ random variables, with nonnegative weights
$Y_i$, motivated by Breiman's conjecture.

We introduce the class $\ISSV$ of
$\II$-\emph{subsequentially slowly varying} tails and establish
a subsequential analogue of Darling's maximum--sum
theorem. This provides a \emph{subsequential generalization}
of Breiman's slowly varying theorem and,
as a \emph{corollary}, resolves the slowly varying
case of his conjecture.

We also establish a \emph{subsequential analogue} of
Breiman's nonzero-index theorem, characterizing
the existence of nondegenerate subsequential limits,
distinct in distribution from $X_1,$ by the failure
of slow variation of the integrated tail, for
weight $Y_1$ with tail distributions outside the class $\ISSV$.

Finally, we formulate the \emph{subsequential separation
principle}, prove it under a finite-second-moment
assumption, and show that its general validity
would imply the nonzero-index case of
Breiman's conjecture.
\end{abstract}

\section{Introduction}

Let $(Y_i)_{i\ge1}$ be independent copies of a non-negative random
variable $Y$, and let $(X_i)_{i\ge1}$ be independent
copies of a real random variable $X$, independent of the $(Y_i).$ 
Throughout this paper, we call $(Y_i)$ the \emph{weights} and $(X_i)$ the \emph{marks}, 
in keeping with the terminology of randomly weighted sums and marked point processes.
Define for $n \in \N,$
\begin{equation}\label{eq:Tn}
S_n=\sum_{i=1}^nY_i, \qquad T_n=T_n(X)=\frac{\sum_{i=1}^nX_iY_i}{\sum_{i=1}^nY_i}=\frac{\sum_{i=1}^nX_iY_i}{S_n},
\end{equation}
with the convention that $T_n=0$ on ${S_n=0}$.

In his seminal paper \cite{Breiman1965}, Breiman investigated the asymptotic behaviour of $T_n$, motivated by limit theorems related to the classical arcsine law. He studied the case when 
the tail distribution 
\begin{equation}\label{eq:tailH}
H(t)=\mathbb P(Y>t), \qquad t>0,
\end{equation}
is \emph{regularly varying} at infinity.

\begin{definition}
A positive measurable function $(f(x), \,\, x\ge x_0>0)$ is called regularly varying at infinity with index $\rho \in \R$ ($f \in RV_{\rho}$) if for any $\lambda>0$
$$\lim_{x \to \infty}\frac{f(\lambda x)}{f(x)} = \lambda^\rho.$$
RV function $f$ of index $\rho=0$ is called slowly varying ($f \in SV$).
\end{definition}

The following two theorems provide modern reformulations
of Breiman's classical results, distinguishing the slowly
varying and nonzero-index regularly varying cases.

\begin{theorem}[Breiman, 1965: slowly varying case]
\label{breiman-sv-thm}
The tail distribution $H$ of the weight $Y$ is $SV$
if and only if for every nondegenerate mark $X,$ independent of $Y,$ with $\EE |X| < \infty,$ one has
$$T_n \dto X.$$
\end{theorem}

\begin{theorem}[Breiman, 1965: nonzero-index regularly varying case]
\label{breiman-rv-thm}
The tail distribution $H$ of the weight $Y$ is
$RV_{-\alpha},$ for some
    $\alpha\in(0,1),$
if and only if for every nondegenerate mark $X,$ independent of $Y,$ with $\EE |X| < \infty,$
the sequence $T_n$ converges in
    distribution to a nondegenerate random
    variable $Z_\alpha(X)$ satisfying
    $$
        \mathcal L(Z_\alpha(X))\ne\mathcal L(X).
    $$
\end{theorem}
Here and throughout, we use the standard notation
$\mathcal{L}(\cdot)$ for the law of a random variable.

\begin{remark}\label{rem_differ_RV_SV}
Theorems~\ref{breiman-sv-thm} and
\ref{breiman-rv-thm} are modern reformulations
of Breiman's original results
\cite[Proposition~2 and Theorems~3--4]{Breiman1965}.
In particular, Breiman established the converse
implications jointly in \cite[Theorem~4]{Breiman1965},
rather than as two separate characterizations.

The distinction between the slowly varying and nonzero-index regularly varying cases is also reflected in the limiting representation~\eqref{Z_Xrepresentation} and the maximum-weight characterization established below.
\end{remark}

In the same paper \cite{Breiman1965}, Breiman conjectured that the requirement
of convergence \emph{for every} nondegenerate integrable mark $X$
in the converse direction of his results could be weakened
to convergence \emph{for a single} such mark. This became known as \textbf{Breiman's conjecture}.



A partial resolution of Breiman's conjecture was
obtained by Mason and Zinn~\cite{MasonZinn2005}
under the stronger moment assumption
\[
\EE|X|^p<\infty
\qquad\text{for some }p>2.
\]
Subsequently, Kevei and Mason~\cite{KeveiMason2016}
extended the converse implication to a broader
class of mark distributions, although the conjecture
remained unresolved under the original integrability
assumption. More recently, a preprint by
Lenzi~\cite{Lenzi2026v3} claims to establish
the conjecture in full generality under the assumption
\[
\EE|X|<\infty.
\]

The transition from full-sequence convergence to
subsequential convergence is classical in probability
theory, particularly in the theory of partial attraction
(see~\cite{Feller1971}), which also provides a foundation
for the theory of semistable laws.
In the setting of randomly weighted means $T_n$,
certain aspects of subsequential convergence were
previously investigated by Kevei and
Mason~\cite{KeveiMason2012}.

The aim of the present paper is to develop subsequential
analogues of Theorems~\ref{breiman-sv-thm} and
\ref{breiman-rv-thm}.
For the slowly varying case, we introduce a new class
of distributional tails, defined through an integrated-tail
condition $\ISSV$ (Definition~\ref{def_ISSV}), which strictly contains
the classical slowly varying tails.
We establish a subsequential analogue of Darling's
maximum--sum theorem (Theorem~\ref{thm_subseq-Darling}),
characterizing the dominance of the largest observation
over the sum along deterministic subsequences.
Building on this result, we obtain a necessary and
sufficient condition for the existence of a
deterministic subsequence, depending only on $H$,
along which $T_n(X)$ converges in distribution to
$X$ for every integrable mark $X$.
Thus, our approach extends Breiman's slowly varying
theorem to the subsequential setting
(Theorem~\ref{thm_subsequent_braimanSV})
and, as a consequence, resolves the slowly varying
case of Breiman's conjecture
(Corollary~\ref{cor_breimanSV}).

For the nonzero-index case, we establish a
subsequential analogue of Breiman's theorem (Theorem~\ref{thm_subsequent_breimanRV})
without assuming regular variation of the tail.
More precisely, for $H\notin\ISSV$, we show that
the existence of a nondegenerate subsequential
limit for some integrable nondegenerate mark $X$
is equivalent to the condition
\[
A(u):=\int_0^u H(t)\,dt\notin SV.
\]
Moreover, under this condition, there exists
a deterministic subsequence $n_k\to\infty$,
depending only on $H$, such that
\[
T_{n_k}(X)\dto Z_X
\]
for every nondegenerate integrable $X$,
where $Z_X$ is nondegenerate and
$\Law(Z_X)\ne\Law(X)$.

Finally, we investigate the structure of subsequential limiting laws and propose the \textbf{subsequential separation principle}. We establish this principle under a finite-second-moment assumption and show that its validity for arbitrary integrable marks would imply the nonzero-index case of Breiman's conjecture. Thus, it provides a natural strengthening of the classical conjecture and a direction for further research.

\section{$\II$-subsequential slow variation and a subsequential analogue of
Theorem~\ref{breiman-sv-thm}}

The slowly varying case of Breiman's theorem
(Theorem~\ref{breiman-sv-thm}) is closely related
to the classical problem of the influence of the
maximum summand on the sum of independent random
variables, studied by Darling \cite{Darling1952}.
Let
$$
    M_n=\max_{1\le i\le n}Y_i,
    \qquad
    S_n=\sum_{i=1}^n Y_i.
$$
For i.i.d. nonnegative weights with an unbounded
distribution, the classical maximum-dominance
characterization states that
\begin{equation}\label{darling-sv}
    H\in SV
    \quad\Longleftrightarrow\quad
    \frac{M_n}{S_n}\pto 1.
\end{equation}

This phenomenon provides a direct probabilistic
explanation of Breiman's slowly varying theorem.
Indeed, let $I_n$ be an index at which the maximum
$M_n$ is attained, chosen independently of the marks.
On the event $S_n>0$, we have
\[
    T_n-X_{I_n}
    =\sum_{i\ne I_n}
       \frac{Y_i}{S_n}(X_i-X_{I_n}).
\]
Since the marks are independent of the weights,
the convergence $M_n/S_n\pto1$
implies $T_n-X_{I_n}\pto0$ for
every integrable mark $X$. Moreover,
$X_{I_n}\stackrel{d}{=}X$, and consequently
$T_n\dto X$.

Our aim in this section is to generalize the notion of
slow variation in order to establish a subsequential
analogue of Darling's result \eqref{darling-sv},
which in turn leads to a subsequential analogue
of Theorem~\ref{breiman-sv-thm}.

\subsection{An integrated tail-level functional}

Let $Y,Y_1,Y_2,\ldots$ be independent identically distributed
nonnegative random variables with unbounded support and tail distribution $H$ given by~\eqref{eq:tailH}. Then $H$ is nonincreasing and right-continuous,
$H(x)>0$ for all $x\ge0$, and $H(x)\downarrow0$ as $x\to\infty$.

Define the decreasing tail quantile
\begin{equation}\label{defQ(p)}
  Q(p)=\inf\{x\ge0:H(x)\le p\},\qquad 0<p<1.
\end{equation}

Now we can define an integrated tail-level functional.
\begin{definition}\label{integrated_funct}

For $t>0$, define
\begin{equation}\label{Rt-def}
R_t(u)
=
\frac{t}{Q(u/t)}
\int_{u/t}^{1}Q(v)\,dv,
\qquad 0<u<tH(0).
\end{equation}
For $A>1$ and $t>A/H(0)$, define the
corresponding integrated tail-level functional
\begin{equation}\label{I-def}
\II_H(t;A)
=
\sup_{A^{-1}\le u\le A}R_t(u).
\end{equation}
The functional is well-defined for $A>1$ and
$t>A/H(0)$, since $Q(u/t)>0$ for every
$u\in[A^{-1},A]$.
\end{definition}

For $0<p<1$ and $x=Q(p)$, we have
$H(y)>p$ for $y<x$ and $H(y)\le p$
for $y\ge x$.
Thus, by Tonelli's theorem and the
layer-cake representation,
\begin{equation}\label{layercakeQ}
\int_p^1 Q(v)\,dv
=\int_0^x\bigl(H(y)-p\bigr)\,dy
=\int_0^x H(y)\,dy-px.
\end{equation}
Equality \eqref{layercakeQ} allows us to rewrite
the functional~\eqref{I-def} in the following form:
\begin{equation}\label{I-def_tail}
\II_H(t;A)
=
\sup_{A^{-1}\le u\le A}
\left(
\frac{t}{Q(u/t)}
\int_0^{Q(u/t)} H(y)\,dy-u
\right).
\end{equation}


The following lemma establishes two elementary useful properties
of $R_t$, defined by~\eqref{Rt-def}.

\begin{lemma}\label{Rt-properties-lemma}
The following statements hold.

\emph{(i)} For $s,t>0$ and $u>0$, whenever both sides
are defined,
\begin{equation}\label{eq:Rt-rescaling}
R_s(u)=\frac{s}{t}R_t\left(\frac{tu}{s}\right).
\end{equation}

\emph{(ii)} For $0<u<v<tH(0)$,
\begin{equation}\label{eq:Rt-level}
R_t(v)\ge R_t(u)-(v-u).
\end{equation}
\end{lemma}

\begin{proof}[Proof of Lemma~\ref{Rt-properties-lemma}]
Statement (i) follows directly from the definition~\eqref{Rt-def}.


To prove (ii) one notes that for $0<u<v<tH(0),$ monotonicity of $Q$
gives
\[
t\int_{v/t}^{1}Q(w)\,dw
\ge
t\int_{u/t}^{1}Q(w)\,dw
-(v-u)Q(u/t).
\]
Since $Q(v/t)\le Q(u/t)$ and the integral
is nonnegative, we obtain
\[
R_t(v)
=\frac{t}{Q(v/t)}
  \int_{v/t}^{1}Q(w)\,dw
\ge
\frac{t}{Q(u/t)}
  \int_{v/t}^{1}Q(w)\,dw
\ge R_t(u)-(v-u).
\]
\end{proof}
\begin{remark}
Inequality~\eqref{eq:Rt-level} provides a one-sided
Lipschitz-type bound for $R_t$. In particular, if
$R_t(u)\ge\delta>0$, then
\[
R_t(v)\ge\frac{\delta}{2}
\]
for every $v\in[u,u+\delta/2]$ within the domain
of $R_t$.
\end{remark}

One of the motivations for introducing the integrated
tail-level functional is that it provides a
characterization of slow variation for the tail
distributions of nonnegative random variables.
The following theorem makes this connection precise.

\begin{theorem}\label{SV_from_I}
Let $H$ be the tail
distribution function of a
nonnegative random variable with unbounded
support. Then
\begin{equation}\label{ITSSV-SV}
H\in SV
\quad\Longleftrightarrow\quad
\lim_{t\to\infty}\II_H(t;A)=0
\qquad\text{for every }A>1.
\end{equation}
\end{theorem}

\begin{proof}[Proof of Theorem~\ref{SV_from_I}]
Suppose first that $H\in SV$. By Karamata's
integral theorem, see for example~\cite{BGT1987,BuldyginPavlenkov2013},
\begin{equation}\label{eq:Karamata-integrated}
\frac{1}{xH(x)}
\int_0^xH(y)\,dy
\longrightarrow1
\qquad (x\to\infty).
\end{equation}

Since $H$ is positive and tends to zero at infinity,
we have $Q(p)\to\infty$ as $p\downarrow0$.
By the definition of $Q$ the monotonicity and right-continuity
of $H$, for every fixed $\lambda\in(0,1)$ and all
sufficiently small $p>0$,
\[
H(Q(p))\le p<H(\lambda Q(p)).
\]
Consequently, slow variation gives
\[
1\ge\frac{H(Q(p))}{p}
>\frac{H(Q(p))}{H(\lambda Q(p))}
\longrightarrow1,
\]
and hence
\begin{equation}\label{eq:quantile-tail-equivalence}
\frac{H(Q(p))}{p}\longrightarrow1
\qquad (p\downarrow0).
\end{equation}

Now fix $A>1$. For $t>A/H(0)$ and
$u\in[A^{-1},A]$, put $x=Q(u/t)>0$.
By~\eqref{I-def_tail},
\begin{equation}\label{rturepresent}
\begin{aligned}
R_t(u)
&= \frac{t}{Q(u/t)}
   \int_0^{Q(u/t)} H(y)\,dy-u \\
&= tH(x)
\left(
\frac{\int_0^x H(y)\,dy}{xH(x)}-1
\right)
+(tH(x)-u).
\end{aligned}
\end{equation}
We aimed to show that both terms in the right hand side of~\eqref{rturepresent} converge
to zero uniformly
for $u\in[A^{-1},A].$ For the first term, 
since $Q$ is nonincreasing,
\[
x=Q(u/t)\ge Q(A/t)\longrightarrow\infty, \quad (t \to \infty),
\]
uniformly for $u\in[A^{-1},A]$.
Moreover, $0\le tH(x)\le u\le A$.
Hence, by \eqref{eq:Karamata-integrated},
\[
\sup_{A^{-1}\le u\le A}
\left|
tH(x)
\left(
\frac{\int_0^x H(y)\,dy}{xH(x)}-1
\right)
\right|
\le
A\sup_{x\ge Q(A/t)}
\left|
\frac{\int_0^x H(y)\,dy}{xH(x)}-1
\right|
\longrightarrow0, \quad (t \to \infty).
\]
Similarly, for the second term, by \eqref{eq:quantile-tail-equivalence},
\[
\begin{aligned}
\sup_{A^{-1}\le u\le A}|tH(x)-u|
&=\sup_{A^{-1}\le u\le A}
u\left|\frac{H(Q(u/t))}{u/t}-1\right|\\
&\le A\sup_{0<p\le A/t}
\left|\frac{H(Q(p))}{p}-1\right|
\longrightarrow0.
\end{aligned}
\]

Therefore, both terms on the right-hand side
of~\eqref{rturepresent} converge uniformly to zero
for $u\in[A^{-1},A]$. Consequently,
\[
\lim_{t\to\infty}\II_H(t;A)=0,
\]
which completes the proof of the forward implication.

For the converse implication, suppose that
for every $A>1$
\begin{equation}\label{Itozero}
  \lim_{t\to\infty}\II_H(t;A)=0.
\end{equation}
For $0<a<b<tH(0)$, monotonicity of $Q$
implies
\[
R_t(a)=
\frac{t}{Q(a/t)}
\int_{a/t}^{1}Q(v)\,dv
\ge
\frac{t}{Q(a/t)}
\int_{a/t}^{b/t}Q(v)\,dv
\ge
(b-a)\frac{Q(b/t)}{Q(a/t)}.
\]
Consequently, form~\eqref{Itozero},
\begin{equation}\label{eq:quantile-separation}
\frac{Q(a/t)}{Q(b/t)}
\longrightarrow\infty
\qquad(t\to\infty)
\end{equation}
for every fixed $0<a<b<\infty$.

Suppose now a contradiction, that
$H\notin SV$. Then, since $H$ is nonincreasing,
there exist $\lambda>1$, $\varepsilon>0$
and $x_k\to\infty$ such that
\begin{equation}\label{Hxkcontradict}
  \frac{H(\lambda x_k)}{H(x_k)}
\le(1-\varepsilon), \quad k \in \N
\end{equation}
For $k \in N$ set $t_k=1/H(x_k)$ and choose fixed constants
$a$ and $b$ such that
\[
1-\varepsilon<a<b<1.
\]
Then, by~\eqref{Hxkcontradict},
\[
H(\lambda x_k)\le \frac{1-\varepsilon}{t_k}<\frac{a}{t_k},
\qquad
H(x_k)=\frac{1}{t_k}>\frac{b}{t_k}, \qquad k \in \N.
\]
Therefore, directly from the definition of $Q$, see~\eqref{defQ(p)},
\[
Q(a/t_k)\le\lambda x_k,
\qquad
Q(b/t_k)\ge x_k.
\]
Hence
\[
\frac{Q(a/t_k)}{Q(b/t_k)}
\le\lambda,
\]
contradicting \eqref{eq:quantile-separation}.
Thus $H\in SV$.
\end{proof}

\begin{remark}
Inspection of the proof of Theorem~\ref{SV_from_I} shows that the
uniform convergence condition in~\eqref{ITSSV-SV} can be replaced by
$R_t(u_0)\to 0$ as $t\to\infty$, for some fixed $u_0>0.$ Equivalently,
$$H\in SV
\quad\Longleftrightarrow\quad
\lim_{t\to\infty}\II_H(t;A)=0
\qquad\text{for some }A>1.$$
Nevertheless, we retain the uniform formulation because
it plays a direct probabilistic role in further Theorem~\ref{thm_subseq-Darling}.
\end{remark}


\subsection{$\II$-subsequential slow variation of distributional tails}

Theorem~\ref{SV_from_I} provides necessary and sufficient conditions for $H\in SV$ in terms of the limiting behaviour of the functional $\II$. This naturally suggests a generalization of slowly varying distributional tails obtained by replacing $\,\,\lim\,\,$ with $\,\,\liminf\,\,$ in~\eqref{ITSSV-SV}.

\begin{definition}\label{def_ISSV}
Let $H$ be the tail
distribution function of a
nonnegative random variable with unbounded
support. We say that $H$ is \emph{$\II$-subsequentially slowly varying}, and
write $H\in\ISSV$, if
for any $A>1,$
$$\liminf_{t\to\infty}\II_H(t;A)=0.$$
\end{definition}
\begin{remark}
The notion of $\,\II$-subsequential slow variation is defined
specifically for distributional tails and therefore does not
provide a generalization of the entire class $SV$ of slowly
varying functions, whose elements are arbitrary positive
measurable functions. Rather, $\ISSV$ extends the subclass
of slowly varying distributional tails.
\end{remark}
\begin{remark}
Equivalently to Definition~\ref{def_ISSV}, because $\II_H \ge 0,$  $H\in\ISSV$ if and only if, for every $A>1$,
there exists a sequence $t_k=t_k(A)\to\infty$ such that
\[
\lim_{k\to\infty}\II_H(t_k;A)=0.
\]
\end{remark}

Although the sequence in the above remark may depend on $A$,
we next show that the lower-limit condition defining $\ISSV$
can be realized along a single deterministic sequence of
integer, independent of $A$.

\begin{lemma}\label{lem_common-subsequence}
The condition $H\in\ISSV$ is equivalent to the existence of a strictly increasing sequence of integers $n_k\to\infty$, independent of $A$, such that
\begin{equation}\label{common-subsequence}
\lim_{k \to \infty}\II_H(n_k;A)=0,
\qquad\text{for every }A>1.
\end{equation}
\end{lemma}

\begin{proof}[Proof of Lemma~\ref{lem_common-subsequence}]
Only the forward implication requires proof. Suppose $H\in \ISSV$. 
By the Definition~\ref{def_ISSV} of $\ISSV$, for every
integer $k\geq 2$,
\[
\liminf_{t\to\infty}\II_H(t;k)=0.
\]
Since $\II_H$ is nonnegative, for any $T>0$ there exists
$t>T$ such that $\II_H(t;k)<1/k$. Thus, proceeding
inductively, we may choose a sequence $t_k\to\infty$
satisfying
\[
t_k>t_{k-1}+2,
\qquad
\II_H(t_k;k)<\frac{1}{k}, \qquad k\geq 2.
\]
Moreover, since $\II_H(t;A)$ is nondecreasing in $A$,
for every fixed $A>1$ and all sufficiently large $k$,
\[
0\leq\II_H(t_k;A)
\leq\II_H(t_k;k)<\frac1k.
\]
Consequently, the same sequence $(t_k)$ satisfies
\[
\lim_{k \to \infty}\II_H(t_k;A)=0
\qquad\text{for every }A>1.
\]

We now show that the real sequence $(t_k)$ can be replaced
by a strictly increasing sequence of integers. Set
$n_k=\lfloor t_k\rfloor$. Since $t_k>t_{k-1}+2$,
the sequence $(n_k)$ is strictly increasing and
$n_k\to\infty$. Furthermore,
\[
0\leq t_k-n_k<1,
\qquad
\frac{n_k}{t_k}\longrightarrow1.
\]

Fix $A>1$. By the rescaling identity~\eqref{eq:Rt-rescaling},
for every $k\in\mathbb N$ and $u\in[A^{-1},A]$, we have
\begin{equation}\label{rnkuresk}
  R_{n_k}(u)
=
\frac{n_k}{t_k}
R_{t_k}\left(\frac{t_ku}{n_k}\right).
\end{equation}

Since $t_k/n_k\to1$, for all sufficiently large $k$,
\[
\left\{
\frac{t_ku}{n_k}:
A^{-1}\leq u\leq A
\right\}
=
\left[\frac{t_k}{A n_k},\frac{A t_k}{n_k}\right]
\subseteq[k^{-1},k].
\]

Therefore, by monotonicity of $\II_H,$ Definition~\ref{integrated_funct} and equality~\eqref{rnkuresk}, we obtain
\[
0\leq\II_H(n_k;A)
\leq
\frac{n_k}{t_k}\II_H(t_k;k)
<
\frac{n_k}{k t_k}
\leq\frac1k.
\]
So,
\[
\lim_{k\to\infty}\II_H(n_k;A)=0
\qquad\text{for every }A>1.
\]
The converse implication follows immediately from
the definition of $\ISSV$.
\end{proof}
\begin{example}\label{ex_ISSV_not_SV}
The inclusion $SV\subset\ISSV$ is strict.

Let
\[
a_k=\exp(\exp(k^2)),\qquad k\ge1,
\]
and define
\[
N(x)=\sum_{k\ge1}\ind_{\{x\ge a_k\}},
\qquad
H(x)=\frac{2^{-N(x)}}{1+\log(1+x)},
\quad x\ge0.
\]
Then $H$ is the tail distribution function of a
nonnegative random variable with unbounded support,
and
\[
H\in\ISSV\setminus SV.
\]
\end{example}

\begin{proof}
First, for all sufficiently large $k$,
\[
\frac{H(a_k)}{H(a_k/2)}
=\frac12\,
\frac{1+\log(1+a_k/2)}
     {1+\log(1+a_k)}
\longrightarrow\frac12.
\]
Therefore $H\notin SV$.

To prove that $H\in\ISSV$, put
\[
L(x)=1+\log(1+x),\qquad
x_k=a_k^k,\qquad
t_k=\frac1{H(x_k)}=2^kL(x_k).
\]
Fix $A>1$ and write
\[
q_k(u)=Q(u/t_k),\qquad A^{-1}\le u\le A.
\]
Since
\[
\frac{L(x_k)}{L(a_k)}\longrightarrow\infty,
\qquad
\frac{L(x_k)}{L(a_{k+1})}\longrightarrow0,
\]
we have $a_k<q_k(u)<a_{k+1}$ for all sufficiently
large $k$, uniformly in $u\in[A^{-1},A]$.
Consequently,
\[
L(q_k(u))=\frac{L(x_k)}{u},
\qquad
t_kH(q_k(u))=u.
\]

Put $b_k=q_k(A)$. Then
\[
\frac{\log b_k}{\log a_k}\sim\frac{k}{A}
\longrightarrow\infty.
\]
In particular,
\[
\frac{2^k a_k L(b_k)}{b_k}\longrightarrow0.
\]

Now let $\ell(x)=1/L(x)$. For
$q\in[a_k,a_{k+1})$, we have
\[
0\le
\int_0^q H(y)\,dy
-2^{-k}\int_0^q\ell(y)\,dy
\le a_k.
\]
Since $\ell\in SV$, Karamata's theorem gives
\[
\frac{\int_0^q\ell(y)\,dy}{q\ell(q)}
\longrightarrow1
\qquad(q\to\infty).
\]
It follows that
\[
\sup_{A^{-1}\le u\le A}
\left|
\frac{\int_0^{q_k(u)}H(y)\,dy}
     {q_k(u)H(q_k(u))}-1
\right|
\longrightarrow0.
\]

Finally, by the definition of $R_t$,
\[
R_{t_k}(u)
=u\left(
\frac{\int_0^{q_k(u)}H(y)\,dy}
     {q_k(u)H(q_k(u))}-1
\right).
\]
Hence
\[
\II_H(t_k;A)
=\sup_{A^{-1}\le u\le A}R_{t_k}(u)
\longrightarrow0
\]
for every $A>1$. Thus $H\in\ISSV$.
\end{proof}

\subsection{Subsequential maximum domination for distributions with $\ISSV$ tails}

We now investigate the probabilistic significance of the ISSV
condition in connection with Darling's classical theorem on
the asymptotic dominance of the maximum over the sum of
independent nonnegative random variables.
Our approach builds on the classical ideas of
Darling~\cite{Darling1952} and their development by
Maller and Resnick~\cite{MallerResnick1984},
adapting them to the subsequential setting.

To this end, let
\[
S_n=\sum_{i=1}^n Y_i,
\qquad
M_n=\max_{1\le i\le n}Y_i,
\]
with the convention that $M_n/S_n=0$ on $\{S_n=0\}$.

To the best of our knowledge, a necessary-and-sufficient
subsequential analogue of Darling's maximum--sum theorem
has not previously been established for general
distributional tails. The following theorem provides
such a characterization along arbitrary deterministic
subsequences, without assuming regular variation
or continuity of the tail distribution.


\begin{theorem}\label{thm_subseq-Darling}
Let $H$ be the tail
distribution function of a
nonnegative random variable $Y$ with unbounded
support. Then
for every deterministic sequence of
integers $n_k\to\infty$, the following
are equivalent:

\emph{(i)}
$$\II_H(n_k;A)\to 0 \qquad \text{for every} \qquad A>1;$$

\emph{(ii)}
$$
\frac{M_{n_k}}{S_{n_k}}
\pto 1.
$$
\end{theorem}

\begin{proof}[Proof of Theorem~\ref{thm_subseq-Darling}]
Let $U_1,\ldots,U_n$ be independent uniform
random variables on $(0,1)$, with order statistics
$U_{(1)}<\cdots<U_{(n)}$.
Put $V_n=nU_{(1)}$. Then $V_n$ has density
\begin{equation}\label{dencity_Vn}
f_n(u)=\left(1-\frac un\right)^{n-1}
\ind_{(0,n)}(u),
\end{equation}
and $V_n\dto E$, where $E\sim\operatorname{Exp}(1)$.

From \eqref{defQ(p)}, using the quantile representation, we may take
$Y_i=Q(U_i)$, so that $M_n=Q(V_n/n)$.
On $\{M_n>0\}$, define
\[
Z_n=\frac{S_n-M_n}{M_n},
\]
and set $Z_n=0$ otherwise.
Conditionally on $V_n=u$, the random variables $U_{(2)},\ldots,U_{(n)}$ are distributed as the order statistics of $n-1$ independent uniform random variables on $(u/n,1)$. Consequently, for $n\ge2$, a direct calculation gives
\begin{equation}\label{conditional-mean}
m_n(u):=\EE[Z_n\mid V_n=u]
=\frac{n-1}{n-u}R_n(u),
\qquad 0<u<nH(0).
\end{equation}

Put $W_n=1-M_n/S_n$, with our convention
on zero denominators. Since $0\le W_n\le1$,
\[
\frac{M_{n_k}}{S_{n_k}}\pto1
\quad\Longleftrightarrow\quad
\EE W_{n_k}\to0.
\]

Suppose first that $\II_H(n_k;A)\to0$
for every $A>1$. Fix $A>1$ and put
$I_A=[A^{-1},A]$.
Since $W_n=Z_n/(1+Z_n)\le Z_n$
on $\{M_n>0\}$, conditioning on $V_{n_k}$
and using \eqref{conditional-mean} gives
\[
\EE W_{n_k}
\le
\sup_{u\in I_A}m_{n_k}(u)
+\PP(V_{n_k}\notin I_A)
\le
\frac{n_k-1}{n_k-A}\II_H(n_k;A)
+\PP(V_{n_k}\notin I_A).
\]
Hence
\[
\limsup_{k\to\infty}\EE W_{n_k}
\le 1-e^{-1/A}+e^{-A}.
\]
Letting $A\to\infty$ proves
$\EE W_{n_k}\to0$, and thus (ii).

Conversely, suppose that (ii) holds,
so $\EE W_{n_k}\to0$.
Conditionally on $V_n=u$, the variable $Z_n$
has the same distribution as a sum of $n-1$
independent random variables taking values
in $[0,1]$.
Therefore,
\[
\EE[Z_n^2\mid V_n=u]
\le m_n(u)^2+m_n(u).
\]
By conditional analogue of Cauchy--Schwarz inequality,
\[
\begin{aligned}
m_n(u)^2
&\le
\EE\left[\frac{Z_n}{1+Z_n}\mid V_n=u\right]
\EE[Z_n(1+Z_n)\mid V_n=u]\\
&\le
\EE[W_n\mid V_n=u]\,
\bigl(m_n(u)^2+2m_n(u)\bigr).
\end{aligned}
\]
Consequently,
\begin{equation}\label{conditional-W-bound}
\EE[W_n\mid V_n=u]
\ge\frac{m_n(u)}{m_n(u)+2},
\qquad 0<u<nH(0).
\end{equation}

Suppose, for a contradiction, that (i) fails.
Then, along a further subsequence,
there exist $A>1$, $\delta>0$, and
$u_k\in[A^{-1},A]$ such that
\[
R_{n_k}(u_k)\ge\delta.
\]
By Lemma~\ref{Rt-properties-lemma}(ii),
\[
R_{n_k}(v)\ge\frac{\delta}{2},
\qquad
v\in J_k:=
\left[u_k,u_k+\frac{\delta}{2}\right].
\]
Since the intervals $J_k$ lie in the fixed
compact set
\[
K=[A^{-1},A+\delta/2]\subset(0,\infty),
\]
equality~\eqref{conditional-mean} gives,
for all sufficiently large $k$,
\[
m_{n_k}(v)\ge\frac{\delta}{4},
\qquad v\in J_k.
\]
Moreover, the densities $f_{n_k}$ converge
uniformly to $e^{-v}$ on $K$.
Thus $f_{n_k}(v)\ge c>0$ on $J_k$
for all sufficiently large $k$.

Applying \eqref{conditional-W-bound}, we obtain
\[
\EE W_{n_k}
\ge\int_{J_k}
\frac{m_{n_k}(v)}{m_{n_k}(v)+2}
f_{n_k}(v)\,dv
\ge\frac{c\delta^2}{2(8+\delta)}>0,
\]
contradicting $\EE W_{n_k}\to0$.
Therefore (i) holds.
\end{proof}

\begin{corollary}\label{Darling_corol}
Let $H$ be the tail distribution of a nonnegative
random variable $Y$ with unbounded support. Then
\[
H\in\ISSV
\quad\Longleftrightarrow\quad
\text{there exists }n_k\to\infty
\text{ such that }
\frac{M_{n_k}}{S_{n_k}}
\xrightarrow{\mathbb P}1.
\]
\end{corollary}
\begin{proof}[Proof of Corollary~\ref{Darling_corol}]
The statement follows immediately from
Theorem~\ref{thm_subseq-Darling}
and Lemma~\ref{lem_common-subsequence}.
\end{proof}

\begin{remark}
Classical approaches to domains of attraction and
partial attraction often rely on characteristic
functions, L\'evy--Khintchine representations,
and Tauberian methods.
The proof of Theorem~\ref{thm_subseq-Darling}
takes a different route, based on conditioning
on the smallest uniform order statistic.
The resulting conditional moment estimates,
combined with the Cauchy--Schwarz inequality,
establish both implications directly.
In particular, the converse implication,
from subsequential maximum domination to
the integrated tail condition, is proved
without appealing to existing attraction theorems.
\end{remark}

\subsection{Randomly weighted sums with \(\ISSV\) tails}

We finally connect subsequential maximum domination
with the convergence of randomly weighted
sums $T_n$.

\begin{lemma}\label{lem_maximum-weight}
Let $Y$ be a nonnegative random variable
with unbounded support, and let $X$ be
an integrable, nondegenerate
random variable. Then, for every
deterministic sequence of integers
$n_k\to\infty$,
\[
\frac{M_{n_k}}{S_{n_k}} \pto 1 
\quad\Longleftrightarrow\quad  T_{n_k} \dto X.
\]
\end{lemma}

\begin{proof}[Proof of Lemma~\ref{lem_maximum-weight}]
Suppose first that $M_{n_k}/S_{n_k}\pto1$.
Let $I_n$ be the index of the first maximum
among $Y_1,\ldots,Y_n$.
Since $I_n$ depends only on the weights,
$X_{I_n}\stackrel d=X$.
Therefore,
\[
\EE\left[|T_n-X_{I_n}|\mid Y_1,\ldots,Y_n\right]
\le 2\EE|X|\left(1-\frac{M_n}{S_n}\right),
\]
almost surely on $\{S_n>0\}$.
Consequently,
\[
\EE\left[
|T_{n_k}-X_{I_{n_k}}|
\ind_{\{S_{n_k}>0\}}
\right]\longrightarrow0.
\]
Since $\PP(S_{n_k}=0)\to0$, we obtain
$T_{n_k}-X_{I_{n_k}}\pto0$, and hence
$T_{n_k}\dto X$.

Conversely, suppose that $T_{n_k}\dto X$.
Put $\mu=\EE X$, $V_i=X_i-\mu$, and
\[
W_n=T_n-\mu\ind_{\{S_n>0\}},
\qquad n\in\N.
\]
For $n\in\N$ and $1\le i\le n$, define
\[
p_{n,i}=\frac{Y_i}{S_n}\ind_{\{S_n>0\}}.
\]
Then $W_{n_k}\dto V:=X-\mu$.
Since
\[
W_n=\sum_{i=1}^n p_{n,i}V_i,
\qquad
\sum_{i=1}^n p_{n,i}\le1,
\]
Jensen's inequality gives
\[
\EE(|W_n|-K)_+
\le\EE(|V|-K)_+,
\qquad K>0.
\]
Thus $(W_n)$ is uniformly integrable, and
\begin{equation}\label{eq_centered_L1_conv}
\EE|W_{n_k}|\longrightarrow\EE|V|.
\end{equation}

Let $V'$ be an independent copy of $V$ and put
\[
\kappa
=\EE|V|-\EE\left|\frac{V+V'}2\right|
=\EE\left[
\min\{|V|,|V'|\}\ind_{\{VV'<0\}}
\right]>0,
\]
where strict positivity follows from
$\EE V=0$ and the nondegeneracy of $V$.

On $\{S_n>0\}$, we have
$\sum_{i=1}^n p_{n,i}=1$, and therefore
\[
W_n=\sum_{i,j=1}^n p_{n,i}p_{n,j}
\frac{V_i+V_j}{2}.
\]
By the triangle inequality and independence
of the marks, on $\{S_n>0\}$,
\[
\begin{aligned}
&\EE[|W_n|\mid Y_1,\ldots,Y_n]\\
&\quad\le
\sum_{i,j=1}^n p_{n,i}p_{n,j}
\EE\left|\frac{V_i+V_j}{2}\right|\\
&\quad=
\left(\sum_{i=1}^n p_{n,i}^2\right)\EE|V|
+\left(1-\sum_{i=1}^n p_{n,i}^2\right)
\EE\left|\frac{V+V'}2\right|\\
&\quad=
\EE|V|-\kappa
\left(1-\sum_{i=1}^n p_{n,i}^2\right).
\end{aligned}
\]

For $n\in\N$, define
\[
Q_n=\sum_{i=1}^n p_{n,i}^2
=\frac{\sum_{i=1}^n Y_i^2}{S_n^2}
\ind_{\{S_n>0\}}.
\]
Taking expectations in the preceding
inequality, we obtain
\[
\kappa\EE\left[
(1-Q_n)\ind_{\{S_n>0\}}
\right]
\le
\PP(S_n>0)\EE|V|-\EE|W_n|.
\]
By \eqref{eq_centered_L1_conv} and
$\PP(S_n>0)\to1$, the right-hand side
converges to zero along $\{n_k\}$.
Consequently,
\[
\EE Q_{n_k}\longrightarrow1.
\]
Finally, since
\[
Q_n\le\frac{M_n}{S_n}\le1,
\]
we conclude that
$M_{n_k}/S_{n_k}\pto1$,
as required.
\end{proof}

\begin{remark}\label{rem_degen_case}
The forward implication of
Lemma~\ref{lem_maximum-weight} remains valid
for degenerate $X$. In fact, in this case
$T_n(X)\pto X$ along the full sequence,
regardless of the behaviour of $M_n/S_n$.
Thus, the nondegeneracy assumption is essential
only for the converse implication.
\end{remark}

We are now ready to formulate and prove a subsequential analogue of Theorem~\ref{breiman-sv-thm}.

\begin{theorem}\label{thm_subsequent_braimanSV}
Let $Y$ be a nonnegative random variable
with unbounded support and tail distribution
function $H$. The following statements
are equivalent:

\emph{(i)} $H\in \ISSV$;

\emph{(ii)} There exists a deterministic
strictly increasing sequence of integers
$n_k\to\infty$, depending only on $H$,
such that
\[
T_{n_k}(X)\dto X
\]
for every integrable
random variable $X$ independent of $Y$;

\emph{(iii)} There exist a nondegenerate
integrable random variable
$X$ independent of $Y$ and a
deterministic strictly increasing
sequence of integers $n_k\to\infty$
such that
\[
T_{n_k}(X)\dto X.
\]
\end{theorem}

\begin{proof}[Proof of Theorem~\ref{thm_subsequent_braimanSV}]
Suppose first that $H\in\ISSV$.
By Lemma~\ref{lem_common-subsequence}
and Theorem~\ref{thm_subseq-Darling},
there exists a deterministic strictly
increasing sequence of integers
$n_k\to\infty$ such that
\[
\frac{M_{n_k}}{S_{n_k}}\pto1.
\]
The forward implication of
Lemma~\ref{lem_maximum-weight}
then gives (ii) for every integrable
$X$, including the degenerate case (Remark~\ref{rem_degen_case}).

The implication (ii)$\Rightarrow$(iii)
is immediate.

Now suppose that (iii) holds.
Since $X$ is nondegenerate,
Lemma~\ref{lem_maximum-weight} gives
\[
\frac{M_{n_k}}{S_{n_k}}\pto1.
\]
By Theorem~\ref{thm_subseq-Darling},
this implies $H\in \ISSV$,
establishing (i).
\end{proof}

\begin{remark}
Condition (ii) of Theorem~\ref{thm_subsequent_braimanSV}
does not require $X$ to be nondegenerate;
integrability alone is sufficient.
\end{remark}

\begin{corollary}\label{cor_breimanSV}
Let $Y$ be a nonnegative random variable
with unbounded support and tail distribution
function $H$. The following statements
are equivalent:

\emph{(i)} $H\in SV$;

\emph{(ii)} For every integrable random
variable $X$ independent of $Y$,
\[
T_n(X)\dto X
\qquad\text{as }n\to\infty;
\]

\emph{(iii)} There exists a nondegenerate
integrable random variable $X$
independent of $Y$ such that
\[
T_n(X)\dto X
\qquad\text{as }n\to\infty.
\]
\end{corollary}

\begin{proof}[Proof of Corollary~\ref{cor_breimanSV}]
The statement follows immediately from Theorem~\ref{thm_subseq-Darling},
Theorem~\ref{thm_subsequent_braimanSV} and Theorem~\ref{SV_from_I}, by taking $n_k=k$.
\end{proof}

\section{Subsequential weak limits of $T_n$ and a subsequential analogue of
Theorem~\ref{breiman-rv-thm}}

In this section, we study the possible subsequential
weak limits of $T_n(X)$ for a fixed integrable
random variable $X$.

For each $n \in \N$, let $Y_{n,(1)}\ge\cdots\ge Y_{n,(n)}$ denote the
order statistics of $Y_1,\ldots,Y_n$ in decreasing order, and set
\[
P_{n,j}:=\frac{Y_{n,(j)}}{S_n}\ind_{\{S_n>0\}},
\quad 1\le j\le n,
\qquad P_{n,j}:=0\quad(j>n).
\]
Let
\[
\mathcal P_\downarrow
:=\left\{p=(p_j)_{j\ge1}:
1\ge p_1\ge p_2\ge\cdots\ge0,\quad
\sum_{j\ge1}p_j\le1\right\}
\]
be the space of ranked mass
partitions, equipped with the topology of
coordinatewise convergence.

The following lemma establishes subsequential
convergence of $(P_{n,j})_{j\ge1}$ and $T_n(X)$.
The representation~\eqref{Z_Xrepresentation} follows from
\cite[Proposition~4.7]{Lenzi2026v3}
and is not claimed to be new.
The remaining statements establish the corresponding
compactness, nondegeneracy, and convergence properties
in the subsequential setting.
For completeness, we provide a self-contained proof
of the lemma in Appendix~A.

\begin{lemma}\label{lem_ranked-compactness}
Let $Y$ be nonnegative random variable with unbounded support. Then:

\emph{(i)} for every deterministic sequence $m_k\uparrow\infty$,
there exist a further deterministic subsequence ${n_k}$ and a random element $P=(P_j)_{j\ge1} \in \mathcal P_\downarrow$
such that
\[
(P_{n_k,j})_{j\ge1}\dto(P_j)_{j\ge1}
\qquad\text{in }\mathcal P_\downarrow.
\]
For each prescribed sequence ${m_k}$, the subsequence
${n_k}$ and the law of $P$ can be chosen depending
only on ${m_k}$ and the distribution of $Y$,
independently of $X$.

\emph{(ii)} Along the same subsequence $\{n_k\}$, for every integrable
real-valued $X$ independent of $Y$,
\[
T_{n_k}(X)\dto Z_X,
\]
\begin{equation}\label{Z_Xrepresentation}
  Z_X
:=\sum_{j\ge1}P_jX_j+
\left(1-\sum_{j\ge1}P_j\right)\mathbb EX,
\end{equation}
where $(X_j)_{j\ge1}$ are i.i.d. copies of $X$,
independent of $P$. The series converges
absolutely almost surely.

\emph{(iii)} For every nondegenerate integrable $X$,
\[
Z_X\text{ is nondegenerate}
\quad\Longleftrightarrow\quad
\mathbb P(P_1>0)>0.
\]

\emph{(iv)}
\[
\mathbb EP_{n_k,1}\longrightarrow\mathbb EP_1,
\qquad
\mathbb E\sum_{j\ge1}P_{n_k,j}^{\,2}
\longrightarrow
\mathbb E\sum_{j\ge1}P_j^2.
\]
\end{lemma}

Now we are ready to formulate and prove 
the generalization of Theorem~\ref{breiman-rv-thm}. Define
\[
A(u):=\int_0^u H(t)\,dt
     =\mathbb E(Y\wedge u),
\qquad u>0,
\]
and
\[
Q_n:=\sum_{j\ge1}P_{n,j}^2
=\frac{\sum_{i=1}^nY_i^2}{S_n^2}
\ind_{\{S_n>0\}}.
\]
We shall use the following characterization
due to Fuchs, Joffe and Teugels
\cite[Theorem~5.6]{FuchsJoffeTeugels2002}:
\begin{equation}\label{AinSVcharact}
 A\in SV
\quad\Longleftrightarrow\quad
\lim_{n\to\infty}\mathbb EQ_n=0.
\end{equation}

\begin{theorem}\label{thm_subsequent_breimanRV}
Let $Y$ be a nonnegative random variable with
unbounded support and tail distribution function
$H\notin\ISSV$. The following statements are
equivalent:

\emph{(i)} The function
\[
A(u):=\int_0^u H(t)\,dt,\qquad u>0,
\]
is not slowly varying;

\emph{(ii)} There exists a deterministic strictly
increasing sequence of integers $n_k\to\infty$,
depending only on $H$, such that, for every
nondegenerate integrable random variable $X$
independent of $Y$,
\[
T_{n_k}(X)\dto Z_X,
\]
where $Z_X$ is nondegenerate and
$\Law(Z_X)\ne\Law(X)$;

\emph{(iii)} There exist a nondegenerate
integrable random variable $X$ independent
of $Y$ and a deterministic strictly increasing
sequence of integers $n_k\to\infty$ such that
\[
T_{n_k}(X)\dto Z
\]
for some nondegenerate random variable $Z$.
\end{theorem}

\begin{proof}[Proof of Theorem~\ref{thm_subsequent_breimanRV}]
We first prove (i)$\Rightarrow$(ii).
Suppose that $A\notin SV$. By the characterization~\eqref{AinSVcharact},
$\mathbb EQ_n$ does not converge to zero.
Since $0\le Q_n\le1$, there exist $\delta>0$
and a deterministic strictly increasing sequence
$m_k\to\infty$ such that
\[
\mathbb EQ_{m_k}\ge\delta,
\qquad k\ge1.
\]
By Lemma~\ref{lem_ranked-compactness}(i),
we may pass to a further subsequence, denoted
by $n_k$, such that
\[
(P_{n_k,j})_{j\ge1}\dto P
\qquad\text{in }\mathcal P_\downarrow.
\]
Part (iv) of the same lemma gives
\[
\mathbb E\sum_{j\ge1}P_j^2
=\lim_{k\to\infty}\mathbb EQ_{n_k}
\ge\delta>0.
\]
Consequently, $\mathbb P(P_1>0)>0$.
By Lemma~\ref{lem_ranked-compactness}(ii)--(iii),
for every nondegenerate integrable $X$,
\[
T_{n_k}(X)\dto Z_X,
\]
where $Z_X$ is nondegenerate. Moreover,
the subsequence $\{n_k\}$ depends only on $H$.

It remains to show that
$\Law(Z_X)\ne\Law(X)$.
Suppose, for contradiction, that
$\Law(Z_X)=\Law(X)$ for some
nondegenerate integrable $X$.
Then
\[
T_{n_k}(X)\dto X.
\]
By Theorem~\ref{thm_subsequent_braimanSV},
this implies $H\in\ISSV$, contradicting
our assumption. Thus
$\Law(Z_X)\ne\Law(X)$ for every
nondegenerate integrable $X$, proving (ii).

The implication
(ii)$\Rightarrow$(iii)
is immediate.

Finally, suppose that (iii) holds. Assume, for contradiction, that
$A\in SV.$ Then by~\eqref{AinSVcharact}
\[
\mathbb EQ_n\longrightarrow0.
\]
Since
\[
P_{n,1}^2\le Q_n,
\]
we have $P_{n,1}\pto0.$
Consider given in (iii) subsequence $\{n_k\}$.
By Lemma~\ref{lem_ranked-compactness}(i),
it has a further subsequence $\{b_k\}$ along which
\[
(P_{b_k,j})_{j\ge1}\dto(P_j)_{j\ge1}
\qquad\text{in }\mathcal P_\downarrow.
\]
Since $P_{n,1}\pto0$,
the convergence $P_{b_k,1}\dto P_1$
implies that $P_1=0$ almost surely.
As $0\le P_j\le P_1$ for every $j\ge1$,
we also have $P_j=0$ almost surely
for all $j\ge1$. Then, by~\eqref{Z_Xrepresentation},
$Z=\mathbb EX$ almost surely, contradicting
the nondegeneracy of $Z$.
Therefore $A\notin SV$, completing the proof.
\end{proof}
\begin{remark}
Unlike the slowly varying case considered in
Theorem~\ref{thm_subsequent_braimanSV},
Theorem~\ref{thm_subsequent_breimanRV}
does not recover the corresponding case of
Breiman's conjecture by simply taking
$n_k=k$. Indeed, the theorem guarantees
the existence of a suitable subsequence,
but does not establish convergence along
the full sequence or regular variation
of $H$.

If $T_n(X)\dto Z$ for some nondegenerate
integrable $X$, where $Z$ is nondegenerate
and $\Law(Z)\ne\Law(X)$, our results imply
that $H\notin\ISSV$ and $A\notin SV$.
However, these conditions do not imply
regular variation of $H$. Thus, the
conclusion of Breiman's conjecture requires
additional information arising from
full-sequence convergence.
\end{remark}

\section{Complementary remarks and open problems}
\label{sec_sublimit_families}

\subsection{Limit distributions and quadratic concentration}

Let $\mathcal P(\mathbb R)$ denote the space of Borel probability
measures on $\mathbb R$, equipped with the topology of weak
convergence. For an integrable mark $X$, put $\mu=\mathbb EX$.

\begin{definition}\label{def_sublimits_family}
For a fixed distribution $\Law(X)$ and nonnegative $Y$ with unbounded support and tail distribution
function $H$, define
\[
\mathscr C_H(X)
:=\left\{\nu\in\mathcal P(\mathbb R):
\Law(T_{n_k}(X))\Rightarrow\nu
\text{ for some }n_k\uparrow\infty\right\}.
\]
Furthermore, with 
\[
Q_n:=
\frac{\sum_{i=1}^nY_i^2}{S_n^2}
\ind_{\{S_n>0\}}, \qquad n \in \N,
\] 
put
\[
d_n:=\mathbb EQ_n,\qquad
\mathscr Q(H):=\operatorname{Acc}\bigl((d_n)_{n\ge1}\bigr),
\]
where $\operatorname{Acc}$ denotes the set of subsequential limits.
\end{definition}

For every $n\in\N$, put
\[
r_n=\frac{Y_{n+1}}{S_{n+1}}
\ind_{\{S_{n+1}>0\}}.
\]
Then, for $n\in\N$,
\[
Q_{n+1}=(1-r_n)^2Q_n+r_n^2,
\]
and hence
\[
|Q_{n+1}-Q_n|
\le 2r_n+r_n^2\le 3r_n.
\]
By exchangeability,
\[
\mathbb E r_n
=\frac{\mathbb P(S_{n+1}>0)}{n+1}
\le\frac1{n+1},\qquad n\in\N.
\]
Consequently,
\[
|d_{n+1}-d_n|
\le\frac3{n+1},\qquad n\in\N.
\]
Since $0\le d_n\le1$ and
$|d_{n+1}-d_n|\to0$, the set of
accumulation points of $\{d_n\}$
is the closed interval
\begin{equation}\label{eq_Q_interval}
\mathscr Q(H)=
\left[\liminf_{n\to\infty}d_n,
      \limsup_{n\to\infty}d_n\right].
\end{equation}

An analogue of the interval property in
\eqref{eq_Q_interval} also holds for
$\mathscr C_H(X)$ in the form of connectedness.
\begin{proposition}\label{propos_on_C_H(X)}
$\mathscr C_H(X)$ is a nonempty, compact, and
connected subset of $\mathcal P(\mathbb R)$, equipped
with the bounded-Lipschitz metric.
\end{proposition}

\begin{proof}[Proof of Proposition~\ref{propos_on_C_H(X)}]
Since
\[
\EE|T_n(X)|\leq\mathbb E|X|, \quad n \in \N,
\]
the laws $\mu_n:=\Law(T_n(X))$ form a tight
family. By Prokhorov's theorem, $(\mu_n)$ is
relatively compact, and hence its cluster set
$\mathscr C_H(X)$ is nonempty and compact.

Finally, the cluster set of a relatively compact
sequence $(x_n)$ in a metric space is connected
whenever $d(x_n,x_{n+1})\to0,$ so connectedness
follows from
\[
d_{\mathrm{BL}}(\mu_{n+1},\mu_n)
\leq\mathbb E|T_{n+1}(X)-T_n(X)|
\leq\frac{2\mathbb E|X|}{n+1}
\longrightarrow0.
\]
\end{proof}

Immediately from Proposition~\ref{propos_on_C_H(X)} we have: 
\begin{corollary}\label{cor_C_cardinality}
Under the assumptions of
Proposition~\ref{propos_on_C_H(X)},
\[
|\mathscr C_H(X)|=1
\qquad\text{or}\qquad
|\mathscr C_H(X)|=2^{\aleph_0}.
\]
\end{corollary}
\begin{remark}
If $X$ is degenerate, then
$\mathscr C_H(X)=\{\Law(X)\}$.
Thus the second alternative in
Corollary~\ref{cor_C_cardinality}
is possible only for nondegenerate $X$.
\end{remark}

Corollary~\ref{cor_C_cardinality} establishes a sharp
dichotomy between the full-sequence convergence
considered in the classical Breiman theorem and
the genuinely subsequential regime addressed by
our generalizations: failure of full-sequence
convergence implies the existence of continuum
many distinct subsequential limiting laws.

\subsection{A unified description of the subsequential regimes}

The next result expresses both subsequential Breiman theorems in
terms of $\mathscr C_H(X)$ and $\mathscr Q(H)$.

\begin{theorem}\label{thm_subseq_limit_classification}
Let $Y$ be nonnegative with unbounded support and tail
function $H.$ Let $X$ be nondegenerate and integrable,
and put $\mu=\EE X$. Then
\begin{align}
0\in\mathscr Q(H)
&\quad\Longleftrightarrow\quad
\delta_\mu\in\mathscr C_H(X),\label{eq_CQ_zero}\\
1\in\mathscr Q(H)
&\quad\Longleftrightarrow\quad
\Law(X)\in\mathscr C_H(X)
\quad\Longleftrightarrow\quad H\in\ISSV,\label{eq_CQ_one}\\
\mathscr Q(H)\cap(0,1)\ne\varnothing
&\quad\Longleftrightarrow\quad
\mathscr C_H(X)\text{ contains a nondegenerate law distinct from } \Law(X).
\label{eq_CQ_interior}
\end{align}
Moreover,
\begin{align}
\mathscr Q(H)=\{0\}
&\quad\Longleftrightarrow\quad A\in SV
\quad\Longleftrightarrow\quad
\mathscr C_H(X)=\{\delta_\mu\},\label{eq_CQ_full_zero}\\
\mathscr Q(H)=\{1\}
&\quad\Longleftrightarrow\quad H\in SV
\quad\Longleftrightarrow\quad
\mathscr C_H(X)=\{\Law(X)\},\label{eq_CQ_full_one}\\
\mathscr Q(H)\cap(0,1)\ne\varnothing
&\quad\Longleftrightarrow\quad
A\notin SV\ \text{ and }\ H\notin SV.\label{eq_CQ_interior_analytic}
\end{align}

For any $s\in\mathscr Q(H)\cap(0,1)$, there exists
a deterministic subsequence $\{n_k\}$, depending only
on $H$ and $s$, such that
\[
d_{n_k}\to s
\quad\text{and}\quad
T_{n_k}(X)\xRightarrow{d} Z_X
\]
for every integrable $X$. The subsequence is
independent of $X$, and the limit is nondegenerate
and differs in law from $X$ whenever $X$ is
nondegenerate.
\end{theorem}

\begin{proof}[Proof of Theorem~\ref{thm_subseq_limit_classification}]
Theorem~\ref{thm_subseq_limit_classification} collects the consequences of
Theorem~\ref{thm_subsequent_braimanSV},
Theorem~\ref{thm_subsequent_breimanRV},
Lemma~\ref{lem_ranked-compactness},
and the preceding results.
The characterization of the degenerate limiting law
in~\eqref{eq_CQ_zero} follows directly from
Lemma~\ref{lem_ranked-compactness}.
\end{proof}

\subsection{Full-sequence convergence and the quadratic-weight gap}\label{ssec_full_conv}

Our subsequential results do not, by themselves, provide the
nonzero-index converse in Breiman's conjecture. Indeed, assume
that, for some nondegenerate integrable $X$,
\begin{equation}\label{eq_full_limit_hyp}
T_n(X)\dto Z,
\qquad Z\text{ is nondegenerate},
\qquad\Law(Z)\ne\Law(X).
\end{equation}
Then $\mathscr C_H(X)=\{\Law(Z)\}$, and
\eqref{eq_CQ_zero}--\eqref{eq_CQ_one} imply
\begin{equation}\label{eq_Q_interior_bounds}
0<\liminf_{n\to\infty}d_n
\le\limsup_{n\to\infty}d_n<1.
\end{equation}
However, these inequalities do not imply that $\{d_n\}$ converges.

Nevertheless, for finite-variance marks $X$ the quadratic concentration $\lim_{n \to \infty} d_n$ is
identified directly from the limiting law.

\begin{proposition}\label{prop_quadratic_L2}
Suppose $0<\operatorname{Var}(X)<\infty,$ and
$T_n(X)\dto Z$. Then
\[
\lim_{n\to\infty}d_n
=\frac{\operatorname{Var}(Z)}{\operatorname{Var}(X)}.
\]
\end{proposition}

\begin{proof}[Proof of Proposition~\ref{prop_quadratic_L2}]
By conditional Jensen's inequality, for every $K>0$,
\[
\mathbb E(T_n(X)^2-K)_+
\le \mathbb E(X^2-K)_+.
\]
Since $\mathbb EX^2<\infty$, the family
$\{T_n(X)^2:n\ge1\}$ is uniformly integrable.
Consequently, $T_n(X)\dto Z$ implies
convergence of the first and second moments,
and in particular $\EE Z=\EE X$.

For the second moment
\[
\mathbb ET_n(X)^2
=\operatorname{Var}(X)d_n
+\mu^2\mathbb P(S_n>0).
\]
Since $\mathbb P(S_n>0)\to1$, passing to the
limit gives
\[
\operatorname{Var}(X)\lim_{n\to\infty}d_n
=\mathbb EZ^2-\mu^2
=\operatorname{Var}(Z),
\]
as required.
\end{proof}

Proposition~\ref{prop_quadratic_L2}, together with
the characterization of Fuchs, Joffe and Teugels
\cite[Theorem~5.3]{FuchsJoffeTeugels2002},
\[
d_n\longrightarrow d\in(0,1)
\quad\Longleftrightarrow\quad
H\in RV_{-(1-d)},
\]
establishes the nonzero-index case of Breiman's
conjecture under the additional moment assumption
$\mathbb E X^2<\infty$.

\subsection{Open problem}
The natural approach to closing the gap identified
in~\eqref{eq_Q_interior_bounds} leads to the following conjecture.

\begin{conjecture}[Subsequential separation principle]
\label{conj_subsequential_separation}
Let $X$ be a nondegenerate integrable random variable.
Suppose that, for two deterministic sequences
$n_k,m_k\uparrow\infty$,
\[
T_{n_k}(X)\dto Z_1,
\qquad
T_{m_k}(X)\dto Z_2,
\]
and
\[
d_{n_k}\to d_1,
\qquad
d_{m_k}\to d_2.
\]
Then
\[
d_1\ne d_2
\quad\Longrightarrow\quad
\Law(Z_1)\ne\Law(Z_2).
\]
\end{conjecture}

As shown above, the subsequential separation principle
holds under the additional moment assumption
\[
\EE X^2<\infty.
\]
The following example shows that the principle also holds
for a class of integrable random variables with
$\EE X^2=\infty$, although its validity for arbitrary
integrable marks remains open.

\begin{example}
Let $U_\alpha$ be a symmetric $\alpha$-stable random variable,
$1<\alpha<2$, with characteristic function
\[
\mathbb E e^{itU_\alpha}=\exp(-|t|^\alpha).
\]
Let $G$ be a standard Gaussian random variable independent
of $U_\alpha$, and put
\[
X=U_\alpha+\sigma G,\qquad \sigma>0.
\]
Then $\mathbb E|X|<\infty$, $\mathbb E X^2=\infty$,
and the subsequential separation principle holds for $X$.
\end{example}
\begin{proof}
Suppose that, along two deterministic subsequences,
\[
T_{n_k}(X)\xrightarrow{d}Z_1,\qquad
T_{m_k}(X)\xrightarrow{d}Z_2,
\]
and
\[
d_{n_k}\to d_1,\qquad d_{m_k}\to d_2.
\]
By Lemma~\ref{lem_ranked-compactness}, passing to further
subsequences if necessary, we obtain limiting ranked
mass partitions $P^{(1)}$ and $P^{(2)}$.

Since $\mathbb EX=0$, the corresponding limiting
characteristic functions are
\[
\varphi_{Z_i}(t)=
\mathbb E\exp\left\{
-|t|^\alpha\sum_{j\ge1}(P_j^{(i)})^\alpha
-\frac{\sigma^2t^2}{2}
\sum_{j\ge1}(P_j^{(i)})^2
\right\},
\qquad i=1,2.
\]
Put
\[
a_i=\mathbb E\sum_{j\ge1}(P_j^{(i)})^\alpha.
\]
By Lemma~\ref{lem_ranked-compactness},
\[
\mathbb E\sum_{j\ge1}(P_j^{(i)})^2=d_i.
\]
Since both power sums are bounded by $1$ and
$2\alpha>2$, Taylor's expansion gives
\[
\varphi_{Z_i}(t)
=1-a_i|t|^\alpha-\frac{\sigma^2d_i}{2}t^2
+o(t^2),\qquad t\to0.
\]
If $\mathcal L(Z_1)=\mathcal L(Z_2)$, subtracting
these expansions gives
\[
(a_1-a_2)|t|^\alpha
+\frac{\sigma^2}{2}(d_1-d_2)t^2=o(t^2).
\]
Since $\alpha<2$, this implies $a_1=a_2$,
and consequently $d_1=d_2$, as required.

\end{proof}

More importantly, the nonzero-index regularly varying
case of Breiman's conjecture would follow directly
from this principle, together with the arguments
developed in Subsection~\ref{ssec_full_conv}.
Indeed, suppose that $T_n(X)\dto Z$, where
$Z$ is nondegenerate and $\Law(Z)\ne\Law(X)$.
By the subsequential separation principle,
$(d_n)$ cannot have two distinct accumulation
points, since every subsequence of $T_n(X)$
converges to the same law. Hence $d_n\to d$.
By \eqref{eq_Q_interior_bounds}, $d\in(0,1),$ therefore
$H\in RV_{-(1-d)}.$ 
Thus, we regard the subsequential separation principle
as a natural strengthening of the classical Breiman
conjecture and a promising direction for further
investigation, particularly in light of its recent
claimed resolution by Lenzi~\cite{Lenzi2026v3}.

A related question is whether the quadratic
concentration $d_n=\EE Q_n$ is the most suitable
contrasting functional for distinguishing
subsequential limiting laws.
One may seek \emph{alternative contrasting functionals}
of the normalized weights, depending only on $H$,
whose distinct subsequential limits necessarily
correspond to distinct limiting laws of $T_n(X)$.
Such an approach could lead to a more general
or more tractable formulation of the
subsequential separation principle.

\section*{Acknowledgements}

The author was supported by the Engineering and
Physical Sciences Research Council (EPSRC)
under grant UKRI 2768.

\section*{Appendix A}
\begin{proof}[Proof of Lemma~\ref{lem_ranked-compactness}]

Note that $(P_{m_k,j})_{j\ge1}$ is a random element
of $\mathcal P_\downarrow$ for every $k\ge1$.
Since $\mathcal P_\downarrow$ is compact,
every family of probability measures on it is tight.
In particular, the sequence of probability laws
\[
\left\{
\mathcal L\bigl((P_{m_k,j})_{j\ge1}\bigr)
\right\}_{k\ge1}
\]
is tight. By Prokhorov's theorem, there exist
a further deterministic subsequence $\{n_k\}$
and a random element $P=(P_j)_{j\ge1}$ of
$\mathcal P_\downarrow$ such that
\[
(P_{n_k,j})_{j\ge1}\dto P
\qquad\text{in }\mathcal P_\downarrow.
\]
Since the laws of the ranked weights are determined
entirely by the distribution of $Y$, the subsequence
$\{n_k\}$ and the law of $P$ can be chosen independently
of $X$. This proves (i).

Now let $\pi_n$ be the permutation of $\{1,\ldots,n\}$
that arranges $Y_1,\ldots,Y_n$ in decreasing order,
with smaller indices preceding larger ones in
the case of equal values, so
$$Y_{n,(j)}=Y_{\pi_n(j)}, \qquad 1\le j\le n.$$
Since $(X_i)$ is independent of $(Y_i)$,
the permutation $\pi_n$ is independent of
$(X_i)$. Moreover, as the $X_i$ are i.i.d.,
\[
(X_{\pi_n(1)},\ldots,X_{\pi_n(n)})
\stackrel{d}{=}(X_1,\ldots,X_n),
\]
and the permuted vector remains independent
of $(Y_1,\ldots,Y_n)$. 
Consequently,
\begin{equation}\label{T_n(X)distrXj}
  T_n(X)
=\sum_{j=1}^n P_{n,j}X_{\pi_n(j)}
\stackrel{d}{=}
\sum_{j\ge1}P_{n,j}X_j,
\end{equation}
where $(X_j)_{j\ge1}$ are i.i.d. copies of $X$,
independent of ranked weights $(P_{n_k,j})_{j\ge1}$. 

For $r\ge1$, define
\begin{equation}\label{Znrdenin}
  Z_{n,r}
:=\sum_{j=1}^rP_{n,j}X_j
+\left(1-\sum_{j=1}^rP_{n,j}\right)\mu
\end{equation}
and
\begin{equation}\label{Zrdefin}
 Z_r
:=\sum_{j=1}^rP_jX_j
+\left(1-\sum_{j=1}^rP_j\right)\mu,
\end{equation}
where we set 
$$\mu=\mathbb EX.$$
By already proved (i) and the independence of $(X_j)$ from
$(P_{n_k,j})_{j\ge1}$, for every fixed $r\ge1$,
\[
(P_{n_k,1},\ldots,P_{n_k,r},X_1,\ldots,X_r)
\dto
(P_1,\ldots,P_r,X_1,\ldots,X_r).
\]
The continuous mapping theorem then gives
\begin{equation}\label{ZnkrdtoZr}
 Z_{n_k,r}\dto Z_r.
\end{equation}

\emph{Case 1 (bounded).} Suppose first that $X$ is bounded, so that
$\operatorname{Var}(X)<\infty$, and set
\[
\widetilde T_n(X):=\sum_{j\ge1}P_{n,j}X_j,
\qquad n\in\N.
\]
We shall show that $Z_r$ converges to $Z_X$
in $L^2$ and that $Z_{n_k,r}$ approximates
$\widetilde T_{n_k}(X)$ uniformly in probability
as $r\to\infty$. These properties, together
with \eqref{ZnkrdtoZr}, allow us to apply
the converging-together theorem
\cite[Theorem~3.2]{Billingsley1999}
to obtain $\widetilde T_{n_k}(X)\dto Z_X$,
and hence, by \eqref{T_n(X)distrXj},
$T_{n_k}(X)\dto Z_X$.
To do so note that $(P_{n,j})_{j\ge1}$ is nonincreasing
and $\sum_{j\ge1}P_{n,j}\le1$, therefore for every
integer $r\ge1$,
\[
(r+1)P_{n,r+1}
\le\sum_{j=1}^{r+1}P_{n,j}\le1.
\]
Consequently,
\[
\sum_{j>r}P_{n,j}^2
\le P_{n,r+1}\sum_{j>r}P_{n,j}
\le\frac1{r+1}.
\]
Similarly,
\[
\sum_{j>r}P_j^2
\le P_{r+1}\sum_{j>r}P_j
\le\frac1{r+1},
\qquad r\ge1.
\]
Moreover, since $(X_j-\mu)$ are independent,
centered, and independent of $(P_{n,j})_{j\ge1}$,
we have
\[
\mathbb E\left[
\sum_{j>r}P_{n,j}(X_j-\mu)
\,\middle|\,(P_{n,j})_{j\ge1}
\right]=0
\]
and
\[
\operatorname{Var}\left(
\sum_{j>r}P_{n,j}(X_j-\mu)
\,\middle|\,(P_{n,j})_{j\ge1}
\right)
=
\operatorname{Var}(X)\sum_{j>r}P_{n,j}^2.
\]
Thus, the conditional second moment equals
the conditional variance. Taking expectations,
we obtain
\begin{equation}\label{sumPnjsqupperbound}
  \mathbb E\left|
\sum_{j>r}P_{n,j}(X_j-\mu)
\right|^2
=\operatorname{Var}(X)\,
\mathbb E\sum_{j>r}P_{n,j}^2
\le\frac{\operatorname{Var}(X)}{r+1}.
\end{equation}
Similarly, 
\begin{equation}\label{sumPjsqupperbound}
\mathbb E\left|
\sum_{j>r}P_j(X_j-\mu)
\right|^2
=\operatorname{Var}(X)\,
  \mathbb E\sum_{j>r}P_j^2\\
\le\frac{\operatorname{Var}(X)}{r+1}.
\end{equation}
Now, by \eqref{Z_Xrepresentation} and \eqref{Zrdefin},
\[
Z_X-Z_r=\sum_{j>r}P_j(X_j-\mu).
\]
Applying \eqref{sumPjsqupperbound} gives
\[
\mathbb E|Z_X-Z_r|^2
\le\frac{\operatorname{Var}(X)}{r+1}
\longrightarrow0.
\]
Thus
\begin{equation}\label{Zr-L2}
Z_r\longrightarrow Z_X
\qquad\text{in }L^2.
\end{equation}
Furthermore, \eqref{Znrdenin} gives
\[
\widetilde T_n(X)-Z_{n,r}
=
\sum_{j>r}P_{n,j}(X_j-\mu)
-\mu\ind_{\{S_n=0\}}.
\]
Hence, by \eqref{sumPnjsqupperbound}
and Chebyshev's inequality, for every
$\varepsilon>0$,
\[
\mathbb P\bigl(
|\widetilde T_n(X)-Z_{n,r}|>\varepsilon
\bigr)
\le
\frac{\operatorname{Var}(X)}
{(r+1)\varepsilon^2}
+\mathbb P(Y=0)^n.
\]
Since $Y$ has unbounded support,
$\mathbb P(Y=0)<1$, and therefore, for every
$\varepsilon>0$,
\begin{equation}\label{TnZnr-approx}
\lim_{r\to\infty}\limsup_{k\to\infty}
\mathbb P\bigl(
|\widetilde T_{n_k}(X)-Z_{n_k,r}|
>\varepsilon
\bigr)=0.
\end{equation}
This verifies the approximation condition
of the converging together theorem
\cite[Theorem~3.2]{Billingsley1999}.

Combining \eqref{ZnkrdtoZr},
\eqref{Zr-L2}, and \eqref{TnZnr-approx},
the converging together theorem
\cite[Theorem~3.2]{Billingsley1999}
implies
\[
\widetilde T_{n_k}(X)\dto Z_X.
\]
Finally, since
$\widetilde T_{n_k}(X)\stackrel{d}{=}T_{n_k}(X)$
by \eqref{T_n(X)distrXj}, we conclude that
$T_{n_k}(X)\dto Z_X$.

\emph{Case 2 (unbounded).} Now let $X$ be arbitrary integrable. For fixed $a>0$ set
\[
X^{(a)}:=X\ind_{\{|X|\le a\}},
\qquad
\delta_a:=\mathbb E|X-X^{(a)}|.
\]
Since $X$ is integrable, $\delta_a\to0$
as $a\to\infty$.

Because $Y_i\ge0$, the normalized weights satisfy
\[
W_{n,i}:=\frac{Y_i}{S_n}\ind_{\{S_n>0\}}\ge0,
\qquad
\sum_{i=1}^n W_{n,i}
=\ind_{\{S_n>0\}}\le1.
\]
By the independence of $(X_i)$ and $(Y_i)$, for any $n \in \N$
we obtain
\[
\mathbb E|T_n(X)-T_n(X^{(a)})|
\le
\mathbb E\sum_{i=1}^n
W_{n,i}|X_i-X_i^{(a)}|
=\sum_{i=1}^n
\mathbb EW_{n,i}\,
\mathbb E|X-X^{(a)}|
=\mathbb P(S_n>0)\delta_a
\le\delta_a.
\]
Similarly, using the representation of $Z_X$
and the independence of $(X_j)$ from $P$,
$$
\mathbb E|Z_X-Z_{X^{(a)}}|
\le
\mathbb E\sum_{j\ge1}P_j
|X_j-X_j^{(a)}| +
\mathbb E\left(1-\sum_{j\ge1}P_j\right)
|\mathbb E(X-X^{(a)})|$$
$$\le
\mathbb E\sum_{j\ge1}P_j\,\delta_a
+
\mathbb E\left(1-\sum_{j\ge1}P_j\right)
\delta_a\\
=\delta_a.
$$
Consequently,
\[
\lim_{a\to\infty}\sup_{n\ge1}
\mathbb E|T_n(X)-T_n(X^{(a)})|=0
\]
and
\[
\lim_{a\to\infty}
\mathbb E|Z_X-Z_{X^{(a)}}|=0.
\]
In particular, for every $\varepsilon>0$,
Markov's inequality gives
\[
\lim_{a\to\infty}\limsup_{k\to\infty}
\mathbb P\bigl(
|T_{n_k}(X)-T_{n_k}(X^{(a)})|
>\varepsilon
\bigr)=0.
\]
Since the above proved bounded case gives
\[
T_{n_k}(X^{(a)})\dto Z_{X^{(a)}}
\qquad\text{for every fixed }a>0,
\]
and $Z_{X^{(a)}}\to Z_X$ in $L^1$,
another application of the converging together
theorem \cite[Theorem~3.2]{Billingsley1999}
implies
\[
T_{n_k}(X)\dto Z_X.
\]

Finally, we verify the absolute almost sure
convergence of the series defining $Z_X$.
Since $P_j\ge0$ and $(X_j)_{j\ge1}$ is
independent of $P$, 
\[
\mathbb E\left[
\sum_{j\ge1}P_j|X_j|
\,\middle|\,P\right]
=\mathbb E|X|\sum_{j\ge1}P_j
\le\mathbb E|X|<\infty.
\]
Taking expectations, we obtain
$\mathbb E\sum_{j\ge1}P_j|X_j|<\infty,$
consequently,
$
\sum_{j\ge1}P_j|X_j|<\infty
$
almost surely, which finishes the proof of statement (ii).

To prove (iii) note that if $\mathbb P(P_1>0)=0$,
$P_j=0$ almost surely for every $j$, and hence
$Z_X=\mathbb EX$ almost surely, leading to degenerate limiting law.

Conversely, suppose that
$\mathbb P(P_1>0)>0$ and $X$ is
nondegenerate. By~\eqref{Z_Xrepresentation},
\[
Z_X=P_1X_1+
\sum_{j\ge2}P_jX_j+
\left(1-\sum_{j\ge1}P_j\right)\mu.
\]
Conditioning on $P$ and $(X_j)_{j\ge2}$,
the last two terms are fixed, while $X_1$
remains independent of the conditioning
variables and has the same nondegenerate
distribution as $X$. Therefore, the conditional
distribution of $Z_X$ on the event $\{P_1>0\}$
is nondegenerate. Since the event $\{P_1>0\}$
has positive probability, $Z_X$ cannot be
almost surely constant, finishing proof of statement (iii).

To prove (iv), observe that the coordinate
map $p\mapsto p_1$ is bounded and continuous
on $\mathcal P_\downarrow$, equipped with
the topology of coordinatewise convergence.
Hence, by (i) and the Portmanteau theorem,
\[
\mathbb EP_{n_k,1}\longrightarrow
\mathbb EP_1
\qquad (k\to\infty).
\]
Next, consider
\[
F(p):=\sum_{j\ge1}p_j^2,
\qquad p\in\mathcal P_\downarrow.
\]
For every $r\ge1$, the function
$F_r(p):=\sum_{j=1}^rp_j^2$ is continuous
on $\mathcal P_\downarrow$. Moreover,
\[
0\le F(p)-F_r(p)
=\sum_{j>r}p_j^2
\le p_{r+1}\sum_{j>r}p_j
\le\frac1{r+1}.
\]
Consequently,
\[
\sup_{p\in\mathcal P_\downarrow}
|F(p)-F_r(p)|
\le\frac1{r+1}\longrightarrow0.
\]
Thus $F$ is the uniform limit of continuous
functions and hence is continuous.
Since $0\le F\le1$, the same argument gives
\[
\mathbb E\sum_{j\ge1}P_{n_k,j}^2
\longrightarrow
\mathbb E\sum_{j\ge1}P_j^2.
\]
This completes the proof of Lemma~\ref{lem_ranked-compactness}.
\end{proof}

\end{document}